\documentclass[12pt]{article}

\usepackage{amsmath,amsthm,amssymb}
\usepackage{mathtools}

\allowdisplaybreaks
\usepackage{mathrsfs}
\usepackage{cite}

\usepackage[pdftex,bookmarksopen=true,bookmarks=true,unicode,setpagesize]{hyperref}
\hypersetup{colorlinks=true,linkcolor=black,citecolor=black}

\newtheorem{theorem}{Theorem}
\newtheorem{corollary}[theorem]{Corollary}
\newtheorem{lemma}[theorem]{Lemma}

\theoremstyle{remark}

\theoremstyle{remark}

\theoremstyle{remark}
\newtheorem{remark}[theorem]{Remark}

\newcommand{\R}{\mathbb R}
\newcommand{\N}{\mathbb N}
\newcommand{\di}{\partial}
\newcommand{\la}{\langle}
\newcommand{\ra}{\rangle}

\begin{document}

\vspace{-20mm}
\begin{center}{\Large \bf
Meixner class of  orthogonal polynomials\\
 of a non-commutative monotone L\'evy noise
}
\end{center}

{\large Eugene Lytvynov}\\ Department of Mathematics,
Swansea University, Singleton Park, Swansea SA2 8PP, U.K.;
e-mail: \texttt{e.lytvynov@swansea.ac.uk}\vspace{2mm}

{\large Irina Rodionova}\\ Department of Mathematics,
Swansea University, Singleton Park, Swansea SA2 8PP, U.K.;
e-mail: \texttt{e.lytvynov@swansea.ac.uk}\vspace{2mm}


{\small

\begin{center}
{\bf Abstract}
\end{center}
\noindent 
Let $(X_t)_{t\ge0}$ denote a non-commutative monotone L\'evy process.  Let $\omega=(\omega(t))_{t\ge0}$ denote 
the corresponding monotone L\'evy noise, i.e., formally $\omega(t)=\frac d{dt}X_t$. 
A continuous polynomial of $\omega$ is an element of the corresponding non-commutative $L^2$-space $L^2(\tau)$ that has the form $\sum_{i=0}^n\la \omega^{\otimes i},f^{(i)}\ra$, where $f^{(i)}\in C_0(\R_+^i)$. We denote by $\mathbf{CP}$ the space of all continuous polynomials of $\omega$. For $f^{(n)}\in C_0(\R_+^n)$, the  orthogonal polynomial $\la P^{(n)}(\omega),f^{(n)}\ra$ is defined as the orthogonal projection of the monomial $\la\omega^{\otimes n},f^{(n)}\ra$ onto the subspace of $L^2(\tau)$ that is orthogonal to all continuous polynomials of $\omega$ of order $\le n-1$. We denote by $\mathbf{OCP}$ the linear span  of the  orthogonal polynomials. Each orthogonal polynomial $\la P^{(n)}(\omega),f^{(n)}\ra$ depends only on the restriction of the function $f^{(n)}$ to the set
$\{(t_1,\dots,t_n)\in\R_+^n\mid t_1\ge t_2\ge\dots\ge t_n\}$. 
The orthogonal polynomials allow us to construct a unitary operator $J:L^2(\tau)\to\mathbb  F$, where $\mathbb F$ is an extended monotone Fock space. Thus, we may think of the monotone noise $\omega$ as a distribution of linear operators acting in $\mathbb F$.
 We say that the orthogonal polynomials belong to the Meixner class if $\mathbf{CP}=\mathbf{OCP}$. 
 We prove that each system of orthogonal polynomials from the Meixner class is characterized by two parameters: $\lambda\in\R$ and $\eta\ge0$. In this case, the monotone L\'evy noise has the  representation  $\omega(t)=\di_t^\dag+\lambda\di_t^\dag\di_t+\di_t+\eta\di_t^\dag\di_t\di_t$. Here, $\di_t^\dag$ and $\di_t$ are the (formal) creation and annihilation operators at $t\in\R_+$ acting in $\mathbb F$. 
  } \vspace{2mm}

  \vspace{2mm}


\section{Introduction}

The Meixner class of orthogonal polynomials on $\R$ was originally derived by Meixner \cite{Meixner} as the class of all Sheffer sequences of monic polynomials that are orthogonal with respect to a probability measure on $\R$ with infinite support. They include the Hermite polynomials, the Charlier polynomials, the Laguerre polynomials, the Meixner polynomials of the first kind, and the Meixner polynomials of the second kind (also called the Meixner--Polaczek polynomials).

Let $(p_n)_{n=0}^\infty$ be a polynomial sequence from the Meixner class.
Let us assume that the measure of orthogonality of these polynomials is centered. Let $\di^\dag$ and $\di$ be linear operators acting on polynomials on $\R$ that satisfy $\di^\dag p_n=p_{n+1}$  and $\di p_n=np_{n-1}$ for all $n$. 
 Then there exist three parameters, $\lambda\in\R$,  $\eta\ge0$,  $k>0$, such that 
 \begin{equation}\label{yr7i58}x=\di^\dag+\lambda\di^\dag\di+k\di+\eta\di^\dag\di\di.\end{equation}
  In this formula,  $x$ denotes the operator of multiplication by the variable $x$, considered as a linear operator acting on polynomials.
  
 For each Meixner sequence of polynomials, its measure of orthogonality is infinitely divisible. Thus, the Meixner class is  related to  L\'evy processes.

It appears that the notion of the Meixner class of orthogonal polynomials admits several generalizations. 
The first generalization is related to changing the definition of the operators $\di^\dag$ and $\di$. Recall that, for $q\in[-1,1]$, one defines the $q$-numbers $[n]_q:=1+q+q^2+\dots+q^{n-1}$. Then, the $q$-Meixner class of orthogonal polynomials on $\R$ is defined as the class of the polynomial sequences $(p_n)_{n=0}^\infty$ satisfying formula \eqref{yr7i58} in which $\di^\dag p_n=p_{n+1}$  and $\di p_n=[n]_q\,p_{n-1}$, see  \cite{BW1,BW2} and the references therein.

Furthermore, it is also possible to generalize the notion of the Meixner class of orthogonal polynomials  to the (classical) infinite dimensional setting, as well as to some non-commutative settings.  

Let us briefly describe the extension to the classical infinite dimensional setting, see \cite{BML,Ly2,L2} for details and \cite{AFS,KL,KSSU,L3,NS,Rod} for related topics.
Consider the Gel'fand triple
$$\mathcal D\subset L^2(\R_+,dt)\subset\mathcal D'.$$
Here $\mathcal D$ is the nuclear space of smooth compactly supported functions on $\R_+:=[0,\infty)$ and $\mathcal D'$ is the dual space of $\mathcal D$, where the dual pairing between $\mathcal D'$ and $\mathcal D$ is obtained by continuously extending the inner product in  $ L^2(\R_+,dt)$. A continuous polynomial $P$ of  $\omega\in\mathcal D'$ is a function $P:\mathcal D'\to\R$ of the form
$$P(\omega)=\sum_{i=0}^n\langle \omega^{\otimes i}, f^{(i)}\rangle,\quad\omega\in\mathcal D', $$
where $f^{(i)}\in\mathcal D^{\odot i}$ (i.e., $f^{(i)}$ belongs to the $i$th symmetric tensor power of $\mathcal D$) and $\langle \omega^{\otimes i}, f^{(i)}\rangle$ denotes the dual pairing of $\omega^{\otimes i}\in\mathcal D^{\prime\,\odot i}$ and $f^{(i)}\in\mathcal D^{\odot i}$.
We denote by $\mathbf{CP}$ the space of all continuous polynomials of $\omega$.

Let $\mu$ be a L\'evy white noise measure on $\mathcal D'$. Thus, $\mu$ is a probability measure on $\mathcal D'$ whose Fourier transform has the Kolmogorov representation
\begin{equation}\label{rtr8l}
\int_{\mathcal D'}e^{i\langle \omega,h\rangle}\,d\mu(\omega)=\exp\left(\int_{\R_+}\int_{\R}
(e^{ish(t)}-1-ish(t))\frac1{s^2}\,d\nu(s)\,dt
\right),\quad h\in\mathcal D.\end{equation}
Here $\nu$ is a finite measure on $\R$. We assume that $\nu$ has all moments finite.
We  also assume, for simplicity, that $\nu$ is a probability measure.

For $f^{(n)}\in\mathcal D^{\odot n}$, we 
define the  orthogonal polynomial  $\la P^{(n)}(\omega),f^{(n)}\ra$ as the orthogonal projection of the monomial $\la\omega^{\otimes n},f^{(n)}\ra$ onto the subspace of $L^2(\mathcal D',\mu)$ that is orthogonal to all continuous polynomials of $\omega$ of order $\le n-1$. We denote by $\mathbf{OCP}$ the linear span of all orthogonal polynomials.

By using the orthogonal polynomials, one constructs a unitary operator $$J:L^2(\mathcal D',\mu)\to\mathbb F,$$
 where
$$\mathbb F=\mathbb R\oplus\bigoplus_{n=1}^\infty L^2_{\mathrm{sym}}(\R_+^n,m_n)$$
is an extended symmetric Fock space. Here $m_n$ is a Radon measure on $\R_+^n$ (which depends on the Kolmogorov measure $\nu$ in formula \eqref{rtr8l}) and $L^2_{\mathrm{sym}}(\R_+^n,m_n)$ denotes the subspace of all symmetric functions from $L^2(\R_+^n,m_n)$. The unitary operator $J$ is given by
$$J\la P^{(n)}(\cdot),f^{(n)}\ra=(0,\dots,0,\underbrace{f^{(n)}}_{\text{$n$th place}},0,0,\dots)\in\mathbb F.$$

For $h\in\mathcal D$, let us preserve the notation $\la\omega,h\ra$ for the operator of multiplication by $\la\omega,h\ra$ in $L^2(\mathcal D',\mu)$. In terms of the unitary operator $J$, we may also think of  $\la\omega,h\ra$ as a linear operator in $\mathbb F$.

We say that the orthogonal polynomials $\la P^{(n)}(\omega),f^{(n)}\ra$ belong to the Meixner class if $\mathbf{CP}=\mathbf{OCP}$. Each system of orthogonal polynomials from the Meixner class is characterized by two parameters: $\lambda\in\R$ and $\eta\ge0$. For each choice of such parameters, we have the following equality for the action of the operator $\la\omega,h\ra$ in $\mathbb F$: for each $f\in\mathcal D$:
\begin{align}
\la\omega,h\ra f^{\otimes n}&=h\odot f^{\otimes n}+\lambda n(hf)\odot f^{\otimes(n-1)}\notag\\
&\quad+n\int_{\R_+}h(u)f(u)\,du\, f^{\otimes (n-1)}+\eta n(n-1)(hf^2)\odot f^{\otimes (n-2)}.
\label{r6e7o896}\end{align}
The choice $\lambda=\eta=0$ gives the  infinite dimensional Hermite polynomials with   $\mu$ being  Gaussian white noise measure; the choice $\lambda\ne0$,  $\eta=0$ gives the infinite dimensional Charlier polynomials with $\mu$ being a centered Poisson
random  measure; the choice $|\lambda|=2\sqrt \eta$, $\eta>0$ gives the infinite dimensional Laguerre polynomials with $\mu$ being the centered gamma random measure;  the choice $|\lambda|>2\sqrt \eta$, $\eta>0$  gives the infinite dimensional  Meixner polynomials of the first kind with $\mu$ being a centered negative binomial random measure; and the choice $|\lambda|<2\sqrt \eta$, $\eta>0$  gives the   infinite dimensional  Meixner polynomials of the second kind with $\mu$
 being the Meixner measure on $\mathcal D'$.
 
 Let $\delta_t$ denote the delta function at $t$.
For each $t\in\R_+$, we formally define operators $\di_t^\dag$, $\di_t$ on $\mathbb F$  by 
$$ \di^\dag_t  f^{\otimes n}:=\delta_t\odot f^{\otimes n}, \quad \di_t  f^{\otimes n}:= nf(t) f^{\otimes (n-1)}.$$
(To be more precise, $\di_t^\dag$ is a formal operator, while the operator $\di_t$ is rigorously defined on a subspace of $\mathbb F$.)
Then formula \eqref{r6e7o896} can be formally written in the form
\begin{equation}\label{yuro86}
\omega(t)=\di_t^\dag+\lambda\di_t^\dag\di_t+\di_t+\eta\di_t^\dag\di_t\di_t,
\end{equation}
compare with \eqref{yr7i58}. (Note that $k=1$ in our case since we chose $\nu$ to be a probability measure.)

In non-commutative probability, monomials $\langle\omega,h\rangle$ are replaced with non-commutative operators $\langle\omega,h\rangle$ (so that $\omega(t)$ can be thought of as a non-commutative noise), while the probability measure $\mu$ is replaced by a state on the algebra of polynomials generated by the monomials $\langle\omega,h\rangle$. There are several non-commutative generalizations of  independence. The most studied  one is free independence, see e.g. \cite{NiSp,Voiculescu}. In the framework of free probability, the Meixner class of orthogonal polynomials of a free L\'evy noise was studied in \cite{BL1,BL2}. 
This study led to a  formula similar to    \eqref{yuro86}. 
For other studies of the free Meixner-type L\'evy processes and the free Meixner polynomials on $\R$ we refer to \cite{a1,a2,a3,a4,a5,a6,BB,Bryc,E1,E2,SY}.

Furthermore, in \cite{bozejkolytvynovwysoczanskiCMP2012}, the notion of a non-commutative L\'evy  noise was introduced for the anyon statistics, and in  \cite{BLR}, the corresponding Meixner class of non-commutative orthogonal polynomials was studied.  Quite unexpectedly, this class was again fully described by a formula similar to  \eqref{yuro86}, albeit its meaning was quite different. Note that the Kolmogorov measures $\nu$ of the corresponding anyon noises are the same as in the case of the classical Meixner noises.

In this paper, we will deal with another important example of non-commutative independence: the monotone independence. This notion was introduced and studied by Muraki \cite{M1,M2,M3}, see also \cite{Franz1,Franz2}. The L\'evy processes of the monotone independence were studied in \cite{Franz2,FM}. Note that there is also the related notion of the anti-monotone independence. We will not discuss it in this paper but only mention that, when trivially modified, all  the results of the present paper hold in the anti-monotone case.  

The main result of the paper is a characterization of the Meixner class of orthogonal polynomials of a monotone L\'evy noise. More precisely, let $(X_t)_{t\ge0}$ be a (non-commutative) monotone L\'evy process.  Let $\omega=(\omega(t))_{t\ge0}$ denote 
the corresponding monotone L\'evy noise, i.e., formally $\omega(t)=\frac d{dt}X_t$. 
A continuous polynomial of $\omega$ is an element of the corresponding non-commutative $L^2$-space $L^2(\tau)$ that has the form $\sum_{i=0}^n\la \omega^{\otimes i},f^{(i)}\ra$, where $f^{(i)}\in C_0(\R_+^i)$. We denote by $\mathbf{CP}$ the space of all continuous polynomials of $\omega$. For $f^{(n)}\in C_0(\R_+^n)$, the  orthogonal polynomial $\la P^{(n)}(\omega),f^{(n)}\ra$ is defined as the orthogonal projection of the monomial $\la\omega^{\otimes n},f^{(n)}\ra$ onto the subspace of $L^2(\tau)$ that is orthogonal to all continuous polynomials of $\omega$ of order $\le n-1$. We denote by $\mathbf{OCP}$ the linear span of all orthogonal polynomials. Each orthogonal polynomial $\la P^{(n)}(\omega),f^{(n)}\ra$ depends only on the restriction of the function $f^{(n)}$ to the set
\begin{equation}\label{t7r9}
T_n:=\{(t_1,\dots,t_n)\in\R_+^n\mid  t_1\ge t_2\ge\dots\ge t_n\}.\end{equation}

The orthogonal polynomials allow us to construct a unitary operator $J:L^2(\tau)\to\mathbb  F$, where $\mathbb F$ now denotes  an extended monotone Fock space:
\begin{equation}\label{rtei786807}
\mathbb F=\mathbb R\oplus\bigoplus_{n=1}^\infty L^2(T_n,m_n).\end{equation}
Here $m_n$ is a Radon measure on $T_n$, determined by the Kolmogorov measure $\nu$ of the monotone noise $\omega$.
 By using this operator $J$, we may think of the monotone noise $\omega$ as a distribution of linear operators acting in $\mathbb F$.

 We say that the orthogonal polynomials of $\omega$ belong to the Meixner class if $\mathbf{CP}=\mathbf{OCP}$. 
 We prove that each system of orthogonal polynomials from the Meixner class is again characterized by two parameters: $\lambda\in\R$ and $\eta\ge0$. In this case, the monotone L\'evy noise has 
 again representation \eqref{yuro86}. 
  Here, for a point $t\in\R_+$, $\di_t^\dag$ and $\di_t$ are the (formal) creation and annihilation operators at $t$ acting in  $\mathbb F$. 
It is worth noting  that the corresponding Kolmogorov measures $\nu$ appear to be the same as in the case of free independence. 

Let us mention a drastic difference between the monotone case and all the other cases mentioned above, see Remark~\ref{ye7648o5680} below for details. In the monotone case, it is possible that a continuous monomial $\la\omega^{\otimes n},f^{(n)}\rangle$  is a non-zero element of $L^2(\tau)$ but belongs to the space of all continuous polynomials of order $\le n-1$. A necessary condition for this is that the restriction of the function $f^{(n)}$ to $T_n$ is equal to zero.

The paper is organized as follows. In Section 2, we briefly recall the notion of monotone independence,  define a monotone L\'evy noise and the corresponding non-commutative $L^2$-space. In Section 3, we formulate the main results. Finally, in Section~4, we prove the results.

Among numerous open problems related to the monotone Meixner orthogonal polynomials, let us mention only one: the explicit form of their generating function; compare with the form of the generating function of the Meixner 
orthogonal polynomials in the free setting \cite{BL2}, see also \cite{a1,a3,a4,SY}.

\section{Monotone L\'evy noise}

Let us first recall the notion of monotone independence, cf. \cite{M2}.  Let $\mathcal F$ be a real Hilbert space and let $\mathcal L(\mathcal F)$ denote the space of all continuous linear operators in $\mathcal F$. Let $\Omega\in\mathcal F$ be a unit vector, and define a state  $\tau:\mathcal L(\mathcal F)\to\R$ by 
$\tau(A):=(A\Omega,\Omega)_{\mathcal F}$. Subalgebras (not necessarily unital) $\mathcal A_1,\mathcal A_2,\dots,\mathcal A_r$ of $\mathcal L(\mathcal F)$ are called {\it monotonically independent with respect to $\tau$} if, for any $i<j>k$ and $A\in\mathcal A_i$, $B\in\mathcal A_j$, $C\in\mathcal A_k$,
$$ABC=\tau(B)AC,$$
and any $i_1>\dots>i_m>j<k_1<\dots<k_n$, $A_1\in \mathcal A_{i_1},\dots,A_m\in\mathcal A_{i_m}$, $B\in\mathcal A_j$, $C_1\in\mathcal A_{k_1},\dots,C_n\in\mathcal A_{k_n}$,
$$\tau(A_1\dotsm A_mBC_1\dotsm C_n)=\tau(A_1)\dotsm \tau(A_m)\tau(B)\tau (C_1)\dotsm\tau(C_n).$$
Operators $A_1,\dots,A_r\in\mathcal L(\mathcal F)$ are called {\it monotonically independent with respect to $\tau$} if the subalgebras $\mathcal A_i=\operatorname{l.s.}(A_i^k\mid k\in\mathbb N)$, $i=1,2,\dots,r$, are monotonically independent. Here $\operatorname{l.s.}$ denotes the linear span. 

Let $B_0(\R_+)$ denote the linear space of all measurable bounded functions on $\R_+$ with compact support. We endow $B_0(\R_+)$ with a topology such that  a sequence $(h_n)_{n=1}^\infty$  converges to a function $h$ in $B_0(\R_+)$ if all functions $h_n$ vanish outside a compact set in $\R_+$ and $\sup_{t\in\R_+}|h_n(t)-h(t)|\to0$ as $n\to\infty$.

Let $(\langle\omega,h\ra)_{h\in B_0(\R_+)}$ be a family of operators from $\mathcal L(\mathcal F)$. We assume that 
 the operators $\langle\omega,h\ra$ depend on $h$ linearly,
and if $h_n\to h$ in $B_0(\R_+)$, then $\la\omega,h_n\ra\to\la \omega, h\ra$ strongly in $\mathcal L(\mathcal F)$.  We may formally think of $\omega$ as an operator-valued distribution. 

We will say that $\omega$ is a {\it monotone L\'evy noise} if the following conditions are satisfied.
\begin{itemize}
\item[(i)] Let  $0\le t_0<t_1<t_2<\dots<t_r$ and let  functions $h_1,h_2,\dots,h_r\in B_0(\R_+)$ be such that, for each  $i=1,2,\dots,r$, $h_i=h_i\,\chi_{[t_{i-1},t_i]}$. (Here $\chi_\Delta$ denotes the indicator function of a set $\Delta$.) Then the operators $\la\omega,h_1\ra,\dots,\la\omega,h_r\ra$ are monotonically independent.

\item[(ii)] For any $h_1,\dots,h_n\in B_0(\R_+)$ and any $u>0$,
$$\tau\big(\la\omega,h_1\ra\dotsm \la\omega,h_n\ra\big)=\tau\big(\la\omega,\mathcal S_u h_1\ra\dotsm \la\omega,\mathcal S_uh_n\ra\big).$$
Here, for $h\in B_0(\R_+)$ and $u>0$, we define $\mathcal S_uh\in B_0(\R_+)$ by
$$(\mathcal S_uh)(t):=\begin{cases}0,&\text{if }0\le t<u,\\
h(t-u),&\text{if }t\ge u.\end{cases}$$
\end{itemize}

If $\omega$ is a monotone L\'evy noise, we define, for $t\ge0$, $X_t:=\la\omega,\chi_{[0,t]}\ra$. Then $(X_t)_{t\ge0}$ is a monotone L\'evy process, cf.\ \cite[Section 4]{FM}.

By analogy with \cite{Franz2,FM}, we will now present an explicit construction of a monotone L\'evy noise.  
Let $\nu$ be a probability measure on $\R$ with compact support.
For each $n\in\N$, 
let $T_n$ be defined by \eqref{t7r9} and we denote 
\begin{align*}
S_n:&=\big\{(t_1,s_1,\dots,t_n,s_n)\in(\R_+\times\R)^n\mid (t_1,\dots,t_n)\in T_n\big\},\\
\mathcal F^{(n)}:&=L^2(S_n,dt_1\,d\nu(s_1)\dotsm dt_n\,d\nu(s_n)).
\end{align*}
Let also $\mathcal F^{(0)}:=\R$. We define the {\it monotone Fock space} by 
$\mathcal F:=\bigoplus_{n=0}^\infty \mathcal F^{(n)}$. 
The vector $\Omega=(1,0,0,\dots)\in\mathcal F$ will be called the {\it vacuum}. As usual, we will identify $f^{(n)}\in\mathcal F^{(n)}$ with the corresponding element $(0,\dots,0,\underbrace{f^{(n)}}_{\text{$n$th place}},0,\dots)$ of $\mathcal F$. 
We define the {\it vacuum state on $\mathcal L(\mathcal F)$} by
$\tau(A):=(A\Omega,\Omega)_{\mathcal F}$ for $A\in\mathcal L(\mathcal F)$.

Let  $h\in B_0(\R_+)$. We define a {\it creation operator} $a^+(h)$, a {\it neutral operator} $a^{0}(h)$, and an {\it annihilation operator} $a^-(h)$ as bounded linear operators in $\mathcal F$ that satisfy the following conditions. For the creation operator, we have $a^+(h)\Omega=h$ (here $h(t,s):=h(t)$), and for $f^{(n)}\in\mathcal F^{(n)}$, $n\in\mathbb N$, we have $a^+(h)f^{(n)}\in\mathcal F^{(n+1)}$,  
$$\big(a^+(h)f^{(n)}\big)(t_1,s_1,\dots,t_{n+1},s_{n+1})=h(t_1)f^{(n)}(t_2,s_2,\dots,t_{n+1},s_{n+1}).$$
For the neutral operator, $a^0(h)\Omega=0$ and for each $f^{(n)}\in\mathcal F^{(n)}$, $n\in\mathbb N$, we have $a^0(h)f^{(n)}\in\mathcal F^{(n)}$,
$$\big(a^0(h)f^{(n)}\big)(t_1,s_1,\dots,t_n,s_n)=h(t_1)s_1f^{(n)}(t_1,s_1,\dots,t_n,s_n).$$
For the annihilation operator, $a^-(h)\Omega=0$ and for $f^{(n)}\in\mathcal F^{(n)}$, $a^-(h)f^{(n)}\in\mathcal F^{(n-1)}$ and
$$\big(a^-(h)f^{(n)}\big)(t_1,s_1,\dots,t_{n-1},s_{n-1})=\int_{t_1}^\infty\int_\R
 h(u)f^{(n)}(u,v,t_1,s_1,\dots,t_{n-1},s_{n-1})\,d\rho(v)\,du.$$
As easily seen, the annihilation operator $a^-(h)$ is the adjoint of the creation operator $a^+(h)$, while the neutral operator $a^0(h)$
 is self-adjoint. Thus, for each $h\in B_0(\R_+)$, we define a self-adjoint operator
 $$\la\omega,h\ra:=a^+(h)+a^0(h)+a^-(h).$$
Note that $\tau(\la\omega,h\ra)=0$. It can be easily checked that $\omega$ is a monotone L\'evy noise. In fact, it follows from \cite{FM} that we have just described essentially all  centered monotone L\'evy noises up to equivalence. (To construct all centered monotone L\'evy processes, one needs to assume that the measure $\nu$ is finite rather than probability)

Let $\mathbf A$ denote the real algebra generated by the operators $(\la\omega,h\ra)_{h\in B_0(\R_+)}$ and the identity operator in $\mathcal F$. We define an inner product on $\mathbf A$ by
$$(P_1,P_2)_{L^2(\tau)}:=\tau(P_2^*P_1)=(P_1\Omega,P_2\Omega)_{\mathcal F},\quad P_1,P_2\in\mathbf A.$$
Let $\widetilde {\mathbf A}:=\{P\in\mathbf A\mid(P,P)_{L^2(\tau)}=0\}$. We define the non-commutative $L^2$-space $L^2(\tau)$ as the completion  of the quotient space $\mathbf A /\widetilde{\mathbf A}$ with respect to the norm generated by the scalar product $(\cdot,\cdot)_{L^2(\tau)}$. Elements $P\in\mathbf A$ are considered as representatives of the equivalence classes from
$\mathbf A /\widetilde{\mathbf A}$, and so $\mathbf A$ becomes a dense subspace of $L^2(\tau)$. 

Note that, for each $h\in B_0(\R_+)$,
$$(\la\omega,h\ra P_1,P_2)_{L^2(\tau)}=(P_1,\la\omega,h\ra P_2)_{L^2(\tau)},\quad P_1,P_2\in\mathbf A.$$
Hence, $\mathbf A\ni P\mapsto \la \omega,h\ra P\in L^2(\tau)$ is a well defined linear operator in $L^2(\tau)$, i.e., $\la \omega,h\ra P=0$ for each $P\in\widetilde{\mathbf A}$, see e.g.\ \cite[Ch.~5, Sect.~5, subsec.~2]{BK}. Furthermore, we may extend this operator by continuity to a bounded self-adjoint linear operator in $L^2(\tau)$. With an abuse of notation, we will denote this operator of left multiplication by $\la \omega,h\ra$ in $L^2(\tau)$ by $\la \omega,h\ra$.

\section{The main results}

We will now present the main results of the paper.

\begin{theorem}\label{yur7}
The vacuum vector $\Omega$  is cyclic for the operator family $\big(\la\omega,h\ra\big)_{h\in B_0(\R_+)}$, i.e., the set $\{P\Omega\mid P\in\mathbf A\}$ is dense in $\mathcal F$. Hence, the mapping $\mathbf A\ni P\mapsto IP:=P\Omega\in\mathcal F$ extends by continuity to a unitary operator $I:L^2(\tau)\to\mathcal F$. 
\end{theorem}

Note that the operator of left multiplication by $\la \omega,h\ra$ in $L^2(\tau)$, which we denoted by $\la \omega,h\ra$, is equal to $I^{-1}\la\omega,h\ra I$, in the latter expression the operator $\la\omega,h\ra$ acting in $\mathcal F$.

For any $(h_1,\dots,h_n)\in B_0(\R_+)^n$, the function 
$$(h_1\otimes\dots\otimes h_n)(t_1,\dots,t_n)=h(t_1)\dotsm h(t_n)$$
 belongs to $B_0(\R_+^n)$.
As easily seen, we can extend the mapping
$$B_0(\R_+)^n\ni (h_1,\dots,h_n)\mapsto \la \omega, h_1\ra\dotsm \la \omega, h_n\ra =:\la \omega^{\otimes n}, h_1\otimes\dots\otimes h_n\ra
\in\mathcal L(\mathcal F)
$$
 by linearity and strong continuity  to a mapping
$$ B_0(\R_+^n)\ni f^{(n)}\mapsto \la \omega^{\otimes n},f^{(n)}\ra \in \mathcal L(\mathcal F).$$
Furthermore, we  may think of $\la\omega^{\otimes n},f^{(n)}\ra $ as the element of $L^2(\tau)$ defined by\linebreak  $I^{-1}\la\omega^{\otimes n},f^{(n)}\ra \Omega$.

We will call $\la\omega^{\otimes n},f^{(n)}\ra $
 a  {\it monomial of $\omega$ of order $n$}.  Sums of such operators and (real) constants form the space $\mathbf P$ of {\it polynomials of  $\omega$}.  Since $\mathbf A\subset\mathbf P$, $\mathbf P$ is dense in $L^2(\tau)$.

It is now standard to introduce orthogonal polynomials.  Indeed, we denote by $\mathbf {P}^{(n)}$ the linear space of all polynomials of $\omega$  of order $\le n$.
Let $\overline{\mathbf{P}^{(n)}}$ denote the closure of $\mathbf {P}^{(n)}$ in $L^2(\tau)$, and let $\mathbf{S}^{(n)}:=\overline{\mathbf{P}^{(n)}}\ominus\overline{\mathbf{P}^{(n-1)}}$. 
We thus  get
$L^2(\tau)=\bigoplus_{n=0}^\infty \mathbf{S}^{(n)}$. 

Let $f^{(n)}\in B_0(\R_+^n)$. We denote the orthogonal projection of the monomial $\la \omega^{\otimes n},f^{(n)}\ra$ onto  $\mathbf{S}^{(n)}$ by $\la P^{(n)}(\omega),f^{(n)}\ra$. We define the subspace $\mathbf{OP}$ of $L^2(\tau)$ as the linear span of the identity operator and the orthogonal polynomials $\la P^{(n)}(\omega),f^{(n)}\ra$ with $f^{(n)}\in B_0(\R_+^n)$ (the space of {\it orthogonal polynomials of $\omega$}). 
One can easily check that $\mathbf{OP}$ is dense in  $L^2(\tau)$.

Our next aim is to calculate the $L^2(\tau)$-norm of $\la P^{(n)}(\omega),f^{(n)}\ra$ for $f^{(n)}\in B_0(\R_+^n)$. 
Let $(p_k)_{k=0}^\infty$ denote the system of monic orthogonal polynomials in $L^2(\mathbb R,\nu)$. (If the support of $\nu$ is finite and consists of $N$ points, we set $p_k:=0$ for $k\ge N$.) Hence, $(p_k)_{k=0}^\infty$ satisfy  the recursion formula
\begin{equation}\label{hdtrss}sp_k(s)=p_{k+1}(s)+b_kp_{k}(s)+a_kp_{k-1}(s),\quad k\in\mathbb N_0,\end{equation}
with $p_{-1}(s):=0$, $a_k>0$, and $b_k\in\mathbb R$. (If the support of $\nu$ has $N$ points, $a_k=0$ for $k\ge N$.)
We define
\begin{equation*}
c_k:=\int_{\mathbb R}p_{k-1}(s)^2\,\nu(ds)=a_0a_1\dotsm a_{k-1},\quad k\in\mathbb N,\end{equation*}
where $a_0:=1$.
Note that $c_1=1$ and $c_k=0$ for $k\ge2$
 if and only if the measure $\nu$ is concentrated at one point.

 We denote by $M$ the set of all multi-indices of the form $(l_1,\dots,l_i)\in\mathbb N_0^i$, $i\in\mathbb N$.
 Fix any  $n\in \mathbb N$.  For each $(l_1,\dots,l_i)\in M$, $l_1+\dots+l_i+i=n$, we denote
\begin{multline*}
T^{(l_1,\dots,l_i)}:=\big\{
(t_1,\dots,t_n)\in \R_+^n\mid t_1=t_2=\dots=t_{l_1+1}>\\ >t_{l_1+2}=t_{l_1+3}=\dots=t_{l_1+l_2+2}>\dots>
t_{l_1+l_2+\dots+l_{i-1}+i}=t_{l_1+l_2+\dots+l_{i-1}+i+1}=\dots=t_n
\big\}.
\end{multline*}
The sets $T^{(l_1,\dots,l_i)}$ with $(l_1,\dots,l_i)\in M$, $l_1+\dots+l_i+i=n$, form a set partition of $T_n$.
Consider the bijection
\begin{equation}\label{tydryde}
T^{(l_1,\dots,l_i)}\ni (t_1,\dots,t_n)\mapsto (t_{l_1+1}, t_{l_1+l_2+2},t_{l_1+l_2+l_3+3},\dots, t_n)\in \widetilde T_i,\end{equation}
where $\widetilde T_i:=\{(t_1,\dots,t_i)\in\R_+^i\mid t_1>t_2>\dots>t_i\}$.
Note that the set $T_i\setminus\widetilde T_i$ is of null Lebesgue measure.
We denote by $m^{(l_1,\dots,l_i)}$ the pre-image of the  measure $c_{l_1}\dotsm c_{l_i}dt_1\dotsm dt_i$ on $T_i$ under the mapping \eqref{tydryde}.  We then extend $m^{(l_1,\dots,l_i)}$ by zero to the whole space $T_n$.   We  define a measure $m_n$ on $T_n$ by 
$$ m_n:=\sum_{(l_1,\dots,l_i)\in M,\ l_1+\dots+l_i+i=n}m^{(l_1,\dots,l_i)}.$$

\begin{theorem}\label{f754o} For any $f^{(n)},g^{(n)}\in B_0(\R_+^n)$, $n\in\N$, we have
$$\big(\la P^{(n)}(\omega),f^{(n)}\ra,\la P^{(n)}(\omega),g^{(n)}\ra \big)_{L^2(\tau)} =(f^{(n)},g^{(n)})_{L^2(T_n,m_n)}.$$
\end{theorem}
 
 We define a Hilbert space
$\mathbb F$ by \eqref{rtei786807}.
We will preserve the notation $\Omega$ for the vacuum vector $(1,0,0,\dots)\in\mathbb F$. 

Note that the space $L^2(T_n,dt_1\dotsm dt_n)$ can be identified with the subspace\linebreak $L^2(T^{(0,\dots,0)},m_n)$ of $L^2(T_n,m_n)$.
 Hence, the space $\mathbb F$ contains the subspace $$\R\oplus\bigoplus _{n=1}^\infty L^2(T_n,dt_1\dotsm dt_n),$$ which is a monotone Fock space. So it is natural to call $\mathbb F$ an {\it extended monotone Fock space}.

Note that Theorem~\ref{f754o} implies, in particular, that $\la P^{(n)}(\omega),f^{(n)}\ra\in L^2(\tau)$ is completely determined by the restriction of the function $f^{(n)}\in B_0(\R_+^n)$ to $T_n$, i.e., by $f^{(n)}\in B_0(T_n)$. Note, however, that such a  statement does not hold, in general,  for the monomials $\la \omega^{\otimes n},f^{(n)}\ra$ in $L^2(\tau)$.

Clearly, the set $B_0(T_n)$ is dense in $L^2(T_n,m_n)$. Hence, by Theorem~\ref{f754o}, the mapping 
\begin{equation}\label{yd7i5ei7}
J^{(n)} \la P^{(n)}(\omega),f^{(n)}\ra:=f^{(n)}\in L^2(T_n,m_n)\end{equation}
can be extended by continuity to a unitary operator $J^{(n)}:\mathbf{S}^{(n)}\to L^2(T_n,m_n)$. 
 We can now define a unitary operator $J:L^2(\tau)\to\mathbb F$ whose restriction to each  $\mathbf{S}^{(n)}$ is equal to $J^{(n)}$. (Here $J^{(0)}$ is the identity operator in $\R$.)
Note that we get the following diagram of the unitary operators:
\begin{equation}\label{cdrte64i}
\mathcal F \xleftarrow{I} L^2(\tau)\xrightarrow{J} \mathbb F.\end{equation}

 Again with an abuse of notations, we will  denote by  $\la\omega,h\ra$ the operator $J\la\omega,h\ra J^{-1}$ in $\mathbb F$.

\begin{theorem}\label{yd7i649} For each $h\in B_0(\R_+)$,  consider $\la\omega,h\ra$ as a continuous linear operator acting in $\mathbb F$. Then
$$\la\omega,h\ra=A^+(h)+B^0(h)+B^-(h).$$
Here, $A^+(h)$ is a creation operator: $A^+(h)\Omega=h$ and for $f^{(n)}\in L^2(T_n,m_n)$, $A^+(h)f^{(n)}\in L^2(T_{n+1}m_{n+1})$, 
$$\big(A^+(h)f^{(n)}\big)(t_1,\dots,t_{n+1})=h(t_1)f^{(n)}(t_2,\dots,t_{n+1}),$$
$B^0(h)$ is a neutral operator: $B^0(h)\Omega=0$, for each $f^{(n)}\in L^2(T_n,m_n)$ we have $B^0(h)f^{(n)}\in  L^2(T_n,m_n)$ and for each $(l_1,\dots,l_i)\in M$, $l_1+\dots+l_i+i=n$ and $(t_1,\dots,t_n)\in T^{(l_1,\dots,l_i)}$,
$$ \big(B^0(h)f^{(n)}\big)(t_1,\dots,t_n) =b_{l_1}h(t_1)f^{(n)}(t_1,\dots,t_n),$$
and $B^-(h)$ is an annihilation operator: $B^-(h)\Omega=0$, for $f^{(1)}\in L^2(T_1,m_1)=L^2(\R_+,dt)$
$$B^-(h)f^{(1)}=\int_{\R_+}h(u)f^{(1)}(u)\,du\,\Omega,$$
and for $n\ge2$ and $f^{(n)}\in L^2(T^n,m_n)$, we have $B^-(h)f^{(n)}\in L^2(T_{n-1},m_{n-1})$, and for each $(l_1,\dots,l_i)\in M$, $l_1+\dots+l_i+i=n-1$ and $(t_1,\dots,t_{n-1})\in T^{(l_1,\dots,l_i)}$,
\begin{multline*}
\big(B^-(h)f^{(n)}\big)(t_1,\dots,t_{n-1})\\
=\int_{t_1}^\infty 
h(u)f^{(n)}(u,t_1,\dots,t_{n-1})\,du+a_{l_1+1}h(t_1)f^{(n)}(t_1,t_1,t_2,\dots,t_{n-1}).\end{multline*}
\end{theorem}

Let $C_0(\R_+^n)$ denote the space of all continuous functions on $\R_+^n$ with compact support. Obviously, $C_0(\R_+^n)\subset B_0(\R_+^n)$. If a polynomial
$P=f^{(0)}+\sum_{i=1}^n\la\omega^{\otimes i},f^{(i)}\ra\in\mathbf{P}$ is such that $f^{(i)}\in C_0(\R_+^n)$ for all $i=1,\dots,n$, we call $P$ a {\it a continuous polynomial of $\omega$}. We denote by $\mathbf{CP}$ the space of all continuous polynomials of $\omega$, and by $\mathbf{CP}^{(n)}$ the space of all continuous polynomials of $\omega$ of order $\le n$.
 It follows immediately from the proof of Theorem \ref{yur7} that 
 $\mathbf{CP}^{(n)}$ is dense in $\overline{\mathbf{P}^{(n)}}$. Hence, $\mathbf{CP}$ is dense in $L^2(\tau)$. 
 
 Further we define the subspace $\mathbf{OCP}$ of $L^2(\tau)$ as the linear span of the identity operator and the orthogonal polynomials $\la P^{(n)}(\omega),f^{(n)}\ra$ with $f^{(n)}\in C_0(T_n)$ (the space of {\it orthogonal  polynomials with continuous coefficients}).  Clearly, $\mathbf{OCP}$ is dense in $L^2(\tau)$.
 
 Thus, we have constructed the subspaces $\mathbf{CP}$ and $\mathbf{OCP}$ of $L^2(\tau)$. We will say that the orthogonal polynomials of monotone L\'evy noise $\omega$  belong to the {\it Meixner class} if $\mathbf{CP}=\mathbf{OCP}$.

\begin{theorem}\label{dre6i4e}
The orthogonal  polynomials of monotone L\'evy noise $\omega$  belong to the  Meixner class if and only if there exist $\lambda\in\R$ and $\eta\ge0$ such that, in formula \eqref{hdtrss}, $b_k=\lambda$ for all $k\in\N_0$ and $a_k=\eta$ for all $k\in\N$. In the latter case, for each $h\in B_0(\R_+)$, the operator $\la\omega,h\ra$ in $\mathbb F$ has the following representation:
\begin{equation}\label{cjtre57k89}
 \la\omega,h\ra= A^+(h)+\lambda A^0(h)+A_1^-(h)+\eta A_2^-(h).\end{equation}
Here, $A^+(h)$ is the creation operator
defined in Theorem~\ref{yd7i649}, 
$A^0(h)$ is a neutral operator: $A^0(h)\Omega=0$, and for $f^{(n)}\in L^2(T_n,m_n)$, $A^0(h)f^{(n)}\in L^2(T_n,m_n)$,
$$\big(A^0(h)f^{(n)}\big)(t_1,\dots,t_n)=h(t_1)f^{(n)}(t_1,\dots,t_n),$$
$A^-_1(h)$ is an annihilation operator of the first kind: $A_1^-(h)\Omega=0$ and 
for $f^{(n)}\in L^2(T_n,m_n)$, $A_1^-(h)f^{(n)}\in L^2(T_{n-1},m_{n-1})$ (for $n=1$ the latter space being $\R$),
$$\big(A_1^-(h)f^{(n)}\big)(t_1,\dots,t_{n-1})=\int_{t_1}^\infty 
h(u)f^{(n)}(u,t_1,\dots,t_{n-1})\,du,$$
and $A_2^-(h)$ is an annihilation operator of the second kind: 
$A_2^-(h)=0$ on $\mathbb R\oplus L^2(\R_+,dt)$, and for each $f^{(n)}\in L^2(T_n,m_n)$, $n\ge2$, $A_2^-(h)\in L^2(T_{n-1},m_{n-1})$,
$$\big(A_2^-(h)f^{(n)}\big)(t_1,\dots,t_{n-1})=h(t_1)f^{(n)}(t_1,t_1,t_2,t_3,\dots,t_{n-1}).$$
\end{theorem}

Following \cite{M1}, for functions $f^{(m)}:T_m\to\R$
and $g^{(n)}:T_n\to\R$, we define their {\it monotone tensor product} $f^{(m)}\rhd g^{(n)}$ as the function from $T_{m+n}$ to $\R$ given  by
$$\big(f^{(m)}\rhd g^{(n)}\big)(t_1,\dots,t_{m+n}):=f^{(m)}(t_1,\dots,t_m)g^{(n)}(t_{m+1},\dots,t_{m+n}).$$
This operation is obviously associative. 

\begin{corollary}\label{t896p}
Consider the system of orthogonal polynomials  from the Mexiner class that  corresponds to parameters  $\lambda\in\R$ and $\eta\ge0$. Then, for each $h,h_1,h_2\in C_0(\R_+)$,
$$\la P^{(1)}(\omega),h\ra=\la\omega,h\ra,\quad \la P^{(2)}(\omega),h_1\rhd h_2\ra=\la\omega^{\otimes 2},h_1\otimes h_2\ra-\lambda\la\omega,h_1h_2\ra-\int_{\R_+}h_1(u)h_2(u)\,du,$$
and for $n\ge2$ and $h_1,\dots,h_n\in C_0(\R_+)$, the following recursion formula holds: 
\begin{align*}
&\la P^{(n)}(\omega),h_1\rhd\dotsm\rhd h_n\ra=\la\omega,h_1\ra \la P^{(n-1)}(\omega),h_2\rhd\dotsm\rhd h_n\ra\\
&\quad-\lambda\la P^{(n-1)}(\omega),(h_1h_2)\rhd h_3\rhd\dotsm\rhd h_n\ra-
\la P^{(n-2)}(\omega), \mathcal I(h_1,h_2,h_3)\rhd h_4\rhd \dotsm\rhd h_n\ra\\
&\quad -\eta\la P^{(n-2)}(\omega),(h_1h_2h_3)\rhd h_4\rhd  \dotsm\rhd h_n\ra,
\end{align*}
where the mapping $\mathcal I:C_0(\R_+)^3\to C_0(\R_+)$ is given by
\begin{equation}\label{su5wsu5}
\big(\mathcal I\mathcal (h_1,h_2,h_3)\big)(t):=\int_t^\infty    h_1(u)h_2(u)\,du\, h_3(t).\end{equation}

\end{corollary}

Let us now show that formula \eqref{cjtre57k89} admits a formal interpretation as in formula \eqref{yuro86}. For each $t\in\R_+$, we set formally
$$\di_t^\dag:=A^+(\delta_t),\quad \di_t:=A^-_1(\delta_t),$$
so that, for each $h\in B_0(\R_+)$,
$$A^+(h)=\int_{\R_+} h(t)\di_t^\dag\,dt,\quad A_1^-(h)=\int_{\R_+} h(t)\di_t\,dt.$$
Thus, for $f^{(n)}\in L^2(T_n,m_n)$, 
\begin{align*}\big(\di_t^\dag f^{(n)}\big)(t_1,\dots,t_{n+1})&=\big(\delta_t\rhd f^{(n)}\big)(t_1,\dots,t_{n+1})\\
&=\chi_{[t_2,\infty)}(t)\delta_t(t_1)f^{(n)}(t_2,\dots,t_{n+1})\end{align*}
 and
$$\big(\di_t f^{(n)}\big)(t_1,\dots,t_{n-1})=\chi_{[t_1,\infty)}(t)f^{(n)}(t,t_1,t_2,\dots,t_{n-1}).$$
We then calculate, for $h\in B_0(\R_+)$ and $(t_1,\dots,t_n)\in T_n$,
\begin{align*}
&\bigg(\int_{\R_+}dt
\,  h(t) \di_t^\dag\di_t f^{(n)}\bigg)(t_1,\dots,t_n)=\int_{\R_+}dt\, h(t) \big(\di_t^\dag\di_t f^{(n)}\big)(t_1,\dots,t_n)\\
&\quad =\int_{\R_+}dt\, h(t)\chi_{[t_2,\infty)}(t)\delta_t(t_1)\big(\di_t f^{(n)}\big)(t_2,\dots,t_n)\\
&\quad =\int_{\R_+}dt\, h(t)\chi_{[t_2,\infty)}^2(t)\delta_t(t_1)f^{(n)}(t,t_2,\dots,t_n)\\
&\quad =\int_{t_2}^\infty dt\, h(t)\delta_t(t_1)f^{(n)}(t,t_2,\dots,t_n)\\
&\quad =h(t_1)f^{(n)}(t_1,t_2,\dots,t_n),
\end{align*}
since $t_1\ge t_2$. Thus,
$$\int_{\R_+}dt
\,  h(t) \di_t^\dag\di_t=A^0(h).$$
Analogously, it can be shown that
$$ \int_{\R_+}dt
\,  h(t) \di_t^\dag\di_t\di_t=A_2^-(h).$$
Thus, we can formally write formula \eqref{cjtre57k89} in the form
\begin{equation}\label{yre57i5}
\la\omega,h\ra=\int_{\R_+}h(t)\big(\di_t^\dag+\lambda\di_t^\dag\di_t+\di_t+\eta\di_t^\dag\di_t\di_t\big)\,dt.\end{equation}
If we also formally write
\begin{equation}\label{tyr7i585p}
\la\omega,h\ra=\int_{\R_+} \omega(t)h(t)\,dt,\end{equation}
then formulas \eqref{yre57i5} and \eqref{tyr7i585p} imply \eqref{yuro86}.

\section{Proofs}
Some parts of the proofs below take their ideas from the case of the free Meixner orthogonal polynomials \cite{BL1}. For the reader's convenience, we will still present  self-contained proofs of the results from Section 3.

Below, for  open intervals $\Delta_1,\Delta_2\subset\R_+$, we write $\Delta_1>\Delta_2$ if for any $t_1\in\Delta_1$ and $t_2\in\Delta_2$, we have $t_1>t_2$. This particularly implies that $\Delta_1\cap\Delta_2=\varnothing$.

In the lemma below, $\operatorname{c.l.s.}$ stands for the closed linear span.

\begin{lemma}\label{dyrr968}
For  $n\in\N$, we define  closed subspaces $\mathcal X^{(n)}$, $\mathcal Y^{(n)}$, and $\mathcal Z^{(n)}$ of $\mathcal F$ by
\begin{align*}
&\mathcal X^{(n)}:=\operatorname{c.l.s.}\big\{\Omega,\ \la\omega, h_1 \ra\dotsm \la \omega,h_i\ra \Omega\mid
h_1,\dots,h_i\in B_0(\R_+),\ i\in\{1,\dots,n\}\big\},\\
&\mathcal Y^{(n)}:=\operatorname{c.l.s.}\big\{\Omega,\ (\chi_{\Delta_1}\otimes s^{l_1})\rhd\dotsm\rhd(\chi_{\Delta_i}\otimes s^{l_i})\mid
 (l_1,\dots,l_i)\in M,\\
 &\quad l_1+\dots+l_i+i\le n,\
 \text{$\Delta_1,\dots,\Delta_i\subset\R_+$ are open intervals, $\Delta_1>\Delta_2>\dots>\Delta_i$}
 \big\},\\
 &\mathcal Z^{(n)}:=\operatorname{l.s.}\big\{\Omega,\ f^{(i)}(t_1,\dots,t_i)q_{l_1}(s_1)\dotsm q_{l_i}(s_i)\mid
 (l_1,\dots,l_i)\in M,\ l_1+\dots+l_i+i\le n,\\
 &\quad f^{(i)}\in L^2(T_i,dt_1\dotsm dt_i),\ \text{each $q_{l_j}$ is a polynomial on $\R$ of order $l_j$}\big\}.
\end{align*}
Then $\mathcal X^{(n)}=\mathcal Y^{(n)}=\mathcal Z^{(n)}$.
\end{lemma}

\begin{proof}
By definition, $\mathcal Y^{(n)}\subset \mathcal Z^{(n)}$. Since the Lebesgue measure is non-atomic, it can be easily shown that
\begin{align*}
&\mathcal Y^{(n)}=\operatorname{l.s.}\big\{\Omega,\ f^{(i)}(t_1,\dots,t_i)s_1^{l_1}\dotsm s_i^{l_i}\mid
 (l_1,\dots,l_i)\in M,\ l_1+\dots+l_i+i\le n,\\
 &\quad f^{(i)}\in L^2(T_i,dt_1\dotsm dt_i)\big\}.
\end{align*}
Hence,  $\mathcal Y^{(n)}= \mathcal Z^{(n)}$. Using the definition of the operator $\la\omega,h\ra$ in $\mathcal F$, one can easily show by induction on $n$ that $\mathcal X^{(n)}\subset\mathcal Z^{(n)}$. Thus, to prove the lemma, it suffices to prove the inclusion $\mathcal Y^{(n)}\subset\mathcal X^{(n)}$. We prove this by induction on $n$. The statement is obviously true for $n=1$. Assume that it is true for up to $n$, and let us prove it for $n+1$. Let $(l_1,\dots,l_i)\in M$, $l_1+\dots+l_i+i=n+1$. If $l_1=0$, we get (using the obvious notations):
\begin{align}
&\chi_{\Delta_1}\rhd (\chi_{\Delta_2}\otimes s^{l_2})\rhd\dotsm \rhd(\chi_{\Delta_i}\otimes s^{l_i}) =a^+(\chi_{\Delta_1})(\chi_{\Delta_2}\otimes s^{l_2})\rhd\dotsm \rhd(\chi_{\Delta_i}\otimes s^{l_i})\notag\\
&\quad=\la\omega,\chi_{\Delta_1}\ra(\chi_{\Delta_2}\otimes s^{l_2})\rhd\dotsm \rhd(\chi_{\Delta_i}\otimes s^{l_i})\in\mathcal X^{(n+1)}.\label{rt6ie7869p}
\end{align}
If $l_1\ge1$,
\begin{align}
&(\chi_{\Delta_1}\otimes s^{l_1})\rhd\dotsm\rhd(\chi_{\Delta_i}\otimes s^{l_i})\notag\\
&\quad=
a^0(\chi_{\Delta_1})(\chi_{\Delta_1}\otimes s^{l_1-1})\rhd(\chi_{\Delta_2}\otimes s^{l_2})\rhd\dotsm\rhd(\chi_{\Delta_i}\otimes s^{l_i})\notag\\
&\quad =\la\omega, \chi_{\Delta_1}\ra(\chi_{\Delta_1}\otimes s^{l_1-1})\rhd(\chi_{\Delta_2}\otimes s^{l_2})\rhd\dotsm\rhd(\chi_{\Delta_i}\otimes s^{l_i})\notag\\
&\qquad-\big(a^+(\chi_{\Delta_1})+a^-(\chi_{\Delta_1})\big)(\chi_{\Delta_1}\otimes s^{l_1-1})\rhd(\chi_{\Delta_2}\otimes s^{l_2})\rhd\dotsm\rhd(\chi_{\Delta_i}\otimes s^{l_i}). \label{dye7i}\end{align}
We get
\begin{align}
&a^+(\chi_{\Delta_1})(\chi_{\Delta_1}\otimes s^{l_1-1})\rhd(\chi_{\Delta_2}\otimes s^{l_2})\rhd\dotsm\rhd(\chi_{\Delta_i}\otimes s^{l_i})\notag\\
&\quad =\chi_{\Delta_1}\rhd (\chi_{\Delta_1}\otimes s^{l_1-1})\rhd(\chi_{\Delta_2}\otimes s^{l_2})\rhd\dotsm\rhd(\chi_{\Delta_i}\otimes s^{l_i}).\label{rte6i74}\end{align}
It follows from \eqref{rt6ie7869p} by approximation that the vector on the right hand side of \eqref{rte6i74} belongs to $\mathcal X^{(n+1)}$. Furthermore,
\begin{align}
&a^-(\chi_{\Delta_1})(\chi_{\Delta_1}\otimes s^{l_1-1})\rhd(\chi_{\Delta_2}\otimes s^{l_2})\rhd\dotsm\rhd(\chi_{\Delta_i}\otimes s^{l_i})\notag\\
&\quad=\int_{\R}v^{l_1-1}\,d\nu(v) (g\otimes s^{l_2})\rhd(\chi_{\Delta_3}\otimes s^{l_3})\rhd\dotsm\rhd(\chi_{\Delta_i}\otimes s^{l_i}),\label{r75i}
\end{align}
where
$$g(t):=\chi_{\Delta_2}(t)\int_{\Delta_1\cap(t,\infty)}du=\chi_{\Delta_2}(t)\int_{\Delta_1}du, $$
since $\Delta_2>\Delta_1$. Hence, the vector on the right hand side of \eqref{r75i} belongs to $\mathcal X^{(n-1)}$. Therefore, the vector on the right nand side of \eqref{dye7i} belongs to $\mathcal X^{(n+1)}$.
\end{proof}

\begin{proof}[Proof of Theorem \ref{yur7}] Since the probability measure $\nu$ on $\R$ has compact support, the set of polynomials on $\R$ is dense in $L^2(\R,\nu)$. Therefore, the set $\bigcup_{n=1}^\infty \mathcal Z^{(n)}$ is dense in $\mathcal F$. Hence, by Lemma \ref{dyrr968}, the set $\bigcup_{n=1}^\infty \mathcal X^{(n)}$ is dense in $\mathcal F$, which implies the theorem.
\end{proof}

\begin{lemma}\label{yut9p}
For $(l_1,\dots,l_i)\in $M, let $\mathcal H_{l_1,\dots,l_i}$ denote the following subspace of $\mathcal F$:
$$\mathcal H_{l_1,\dots,l_i}=\big\{f^{(i)}(t_1,\dots,t_i)p_{l_1}(s_1)\dotsm p_{l_i}(s_i)\mid f^{(i)}\in L^2(T_i,dt_1\dotsm dt_i)\big\}.$$
Then 
\begin{equation}\label{dy457i9}\mathcal F=\mathcal F^{(0)}\oplus\bigoplus_{(l_1,\dots,l_i)\in M}\mathcal H_{l_1,\dots,l_i}.\end{equation}
\end{lemma}

\begin{proof}
The statement that the subspaces $\mathcal H_{l_1,\dots,l_i}$ are orthogonal to each other follows from the definition of the scalar product in $\mathcal F$ and the fact that the polynomials $(p_k)_{k=0}^\infty$ are orthogonal in $L^2(\R,\nu)$. Since the polynomials are dense in $L^2(\R,\nu)$, the statement follows.
\end{proof}

We denote $\mathcal H^{(0)}:=\mathcal F^{(0)}$, and for $n\in\mathbb N$ we denote
\begin{equation}\label{crte6ei6}
\mathcal H^{(n)}:=\bigoplus_{\substack{(l_1,\dots,l_i)\in M\\l_1+\dots+l_i+i=n}}\mathcal H_{l_1,\dots,l_i}.\end{equation}
Using Lemma~\ref{yut9p}, we get 
\begin{equation}\label{y69089}
\mathcal F=\bigoplus_{n=0}^\infty \mathcal H^{(n)}.\end{equation}

\begin{lemma}\label{ut96p}
For each $n\in\N_0$, we have $I\big(\mathbf{OP}^{(n)}\big)=\mathcal H^{(n)}$.
\end{lemma}

\begin{proof} In view of \eqref{y69089}, the statement of the lemma is equivalent to the statement 
\begin{equation}\label{esw6e}
I\left(\overline{\mathbf{P}^{(n)}}\right)=\bigoplus_{i=0}^n\mathcal H^{(i)}.
\end{equation}
By \eqref{crte6ei6} and Lemma \ref{dyrr968}, 
$$ \bigoplus_{i=0}^n\mathcal H^{(i)}=\mathcal F^{(0)}\oplus \bigoplus_{\substack{(l_1,\dots,l_i)\in M\\l_1+\dots+l_i+i\le n}}\mathcal H_{l_1,\dots,l_i}=\mathcal Z^{(n)}=\mathcal X^{(n)},$$
which implies \eqref{esw6e}.
\end{proof}

For each $h\in B_0(\R_+)$, we will now represent the neutral operator $a^0(h)$ as a sum of three operators. To this end, we define bounded linear operators $a^{0+}(h)$, $a^{00}(h)$, and $a^{0-}(h)$ in $\mathcal F$ by 
$$a^{0+}(h)\Omega=a^{00}(h)\Omega=a^{0-}(h)\Omega=0$$
and for any $(l_1,\dots,l_i)\in M$ and $f^{(i)}\in L^2(T_i,dt_1\dotsm dt_i)$, \begin{align*}
&a^{0+}(h)f^{(i)}(t_1,\dots,t_i)p_{l_1}(s_1)\dotsm p_{l_i}(s_i)=h(t_1)f^{(i)}(t_1,\dots,t_i)p_{l_1+1}(s_1)p_{l_2}(s_2)\dotsm p_{l_i}(s_i),\\
&a^{00}(h)f^{(i)}(t_1,\dots,t_i)p_{l_1}(s_1)\dotsm p_{l_i}(s_i)=h(t_1)f^{(i)}(t_1,\dots,t_i)b_{l_1}p_{l_1}(s_1)\dotsm p_{l_i}(s_i),\\
&a^{0-}(h)f^{(i)}(t_1,\dots,t_i)p_{l_1}(s_1)\dotsm p_{l_i}(s_i)=h(t_1)f^{(i)}(t_1,\dots,t_i)a_{l_1}p_{l_1-1}(s_1)p_{l_2}(s_2)\dotsm p_{l_i}(s_i).
\end{align*}
In view of \eqref{hdtrss}, we therefore get
\begin{equation}\label{futr8o56}
a^0(h)=a^{0+}(h)+a^{00}(h)+a^{0-}(h).\end{equation}

\begin{lemma}\label{ctre644e}
For any $h_1,\dots,h_n\in B_0(\R_+)$, we have
$$I\la P^{(n)}(\omega),h_1\otimes\dots\otimes h_n\ra=(a^+(h_1)+a^{+0}(h_1))\dotsm
(a^+(h_{n-1})+a^{+0}(h_{n-1}))a^+(h_n)\Omega.$$
\end{lemma}

\begin{proof} Recall that $\la P^{(n)}(\omega),h_1\otimes\dots\otimes h_n\ra$ is the orthogonal projection in $L^2(\tau)$ of the monomial $\la\omega^{\otimes n},h_1\otimes\dots\otimes h_n\ra=\la\omega,h_1\ra\dotsm \la \omega,h_n\ra$ onto $\mathbf{OP}^{(n)}$.  Hence, by Lemma \ref{ut96p},  $I\la P^{(n)}(\omega),h_1\otimes\dots\otimes h_n\ra$, is the orthogonal projection in $\mathcal F$ of the vector $\la\omega,h_1\ra\dotsm \la \omega,h_n\ra\Omega$
onto $\mathcal H^{(n)}$. By \eqref{futr8o56}, for each $h\in B_0(\R_+)$,
\begin{equation}\label{uf77irc}
\la\omega,h\ra=a^+(h)+a^{0+}(h)+a^{00}(h)+a^{0-}(h)+a^-(h).\end{equation}
From here the statement easily follows.
\end{proof}

\begin{lemma}\label{r57r9}
For any $h_1,\dots,h_n\in B_0(\R_+)$, we have
\begin{align*}
I\la P^{(n)}(\omega),h_1\otimes\dots\otimes h_n\ra
 &=\sum_{\substack{(l_1,\dots,l_i)\in M\\l_1+\dots+l_i+i=n}} \big((h_1\dotsm h_{l_1+1})\otimes p_{l_1}\big)\rhd \big((h_{l_1+2}\dotsm h_{l_1+l_2+2})\otimes p_{l_2}\big)\\
 &\qquad\rhd\dots\rhd
\big((h_{l_1+l_2+\dots+l_{i-1}+i}\dotsm h_{n})\otimes p_{l_i}\big).
\end{align*}
\end{lemma}

\begin{proof}By Lemma \ref{ctre644e},
\begin{align*}
&I\la P^{(n)}(\omega),h_1\otimes\dots\otimes h_n\ra
 =\sum_{\substack{(l_1,\dots,l_i)\in M\\l_1+\dots+l_i+i=n}} 
 a^{0+}(h_1)\dotsm a^{0+}(h_{l_1})a^{+}(h_{l_1+1})\\
 & \times a^{0+}(h_{l_1+2})\dotsm a^{0+}
 (h_{l_1+l_2+1})a^+(h_{l_1+l_2+2})\dotsm a^{0+}(h_{l_1+l_2+\dots+l_{i-1}+i})\dotsm a^{0+}(h_{n-1})a^+(h_n)\Omega.
\end{align*}
From here the statement follows.
\end{proof}

\begin{proof}[Proof of Theorem \ref{f754o}]
It suffices to show that, for any $h_1,\dots,h_n\in B_0(\R_+)$,
$$\|\la P^{(n)}(\omega),h_1\otimes\dots\otimes h_n\ra\|_{L^2(\tau)}^2=\|h_1\rhd\dotsm\rhd h_n\|_{L^2(T_n,m_n)}^2,$$
or equivalently
$$\|I\la P^{(n)}(\omega),h_1\otimes\dots\otimes h_n\ra\|_{\mathcal F}^2=\|h_1\rhd\dotsm\rhd h_n\|_{L^2(T_n,m_n)}^2.$$
But the latter formula follows immediately from Lemma \ref{r57r9} and the construction of the measure $m_n$.
\end{proof}

Recall the diagram \eqref{cdrte64i}. We define a unitary operator $U:\mathcal F\to\mathbb F$ by $U:=JI^{-1}$. We will now present en explicit form of the action of $U$. To this end, we recall the orthogonal decomposition \eqref{dy457i9} of $\mathcal F$.

\begin{corollary}\label{utr8o6reds} Let $(l_1,\dots,l_i)\in M$ and let $f^{(i)}\in L^2(T_i,dt_1\dotsm dt_i)$.
Denote
$$F(t_1,s_1,\dots,t_i,s_i)=f^{(i)}(t_1,\dots,t_i)p_{l_1}(s_1)\dotsm p_{l_i}(s_i)\in\mathcal H_{l_1,\dots, l_i}.$$
Let $n=l_1+\dots+l_i+i$.  Define a function $f^{(n)}:T_n\to\R$ by 
$$f^{(n)}(\underbrace{t_1,\dots,t_1}_{\text{$l_1+1$ times}}, \underbrace{t_2,\dots,t_2}_{\text{$l_2+1$ times}},\dots,\underbrace{t_i,\dots,t_i}_{\text{$l_i+1$ times}}):=f^{(i)}(t_1,\dots,t_i)
\quad \text{if }t_1>t_2>\dots>t_i\ge0,$$
and $f^{(n)}(t_1,\dots,t_n)=0$ otherwise. Then $UF=f^{(n)}$. Furthermore, $U\Omega=\Omega$. 
\end{corollary}

\begin{proof} The statement $U\Omega=\Omega$ is trivial. To prove that $UF=f^{(n)}$, it is sufficient to consider the case where $f^{(i)}\in B_0(T_i)$. Then, the function $f^{(n)}$ defined in Corollary \ref{utr8o6reds} belongs to $B_0(T_n)$.
It follows from Lemma~\ref{r57r9} by approximation that
$$I\la P^{(n)}(\omega),f^{(n)}\ra=f^{(i)}(t_1,\dots,t_i)p_{l_1}(s_1)p_{l_2}(s_2)\dotsm p_{l_i}(s_i)=F .$$
Thus, $I^{-1}F=\la P^{(n)}(\omega),f^{(n)}\ra$. Hence, by \eqref{yd7i5ei7},
$UF=JI^{-1}F=f^{(n)}$.
\end{proof}

\begin{proof}[Proof of Theorem \ref{yd7i649}]
Recall formula \eqref{uf77irc}.
Denote, for $h\in B_0(\R_+)$,
$$\alpha^+(h):=a^+(h)+a^{0+}(h),\quad \alpha^0(h):=a^{00}(h),\quad \alpha^-(h):=
a^{0-}(h)+a^-(h),$$
so that 
$$\la\omega,h\ra=\alpha^+(h)+\alpha^0(h)+\alpha^-(h).$$
Recall formula \eqref{crte6ei6}.
It is easy to see that $\alpha^+(h)$ maps $\mathcal H^{(n)}$ into $\mathcal H^{(n+1)}$, $\alpha^0(h)$ maps $\mathcal H^{(n)}$ into itself, and $\alpha^-(h)$ maps
$\mathcal H^{(n)}$ into $\mathcal H^{(n-1)}$. Furthermore, by using Corollary~\ref{utr8o6reds}, one  easily shows that  
$$U\alpha^+(h)U^{-1}=A^+(h),\quad U\alpha^0(h)U^{-1}=B^0(h),\quad U\alpha^-(h)U^{-1}=B^-(h).$$
\end{proof}

\begin{lemma}\label{xtse546} 
Assume that,  in formula \eqref{hdtrss}, $b_k=\lambda$ for all $k\in\N_0$ and some $\lambda\in\R$ and  $a_k=\eta$ for all $k\in\N$ for some $\eta\ge0$. Then formula \eqref{cjtre57k89} holds.
\end{lemma}

\begin{proof}
Immediate from Theorem \ref{yd7i649}.
\end{proof}

For $n\in\mathbb N$ and $i=0,1,\dots,n$, consider a continuous linear operator $R_{i,n}:B_0(\R_+^n)\to B_0(\R_+^i)$. 
(For $i=0$, we set $B_0(\R_+^i):=\R$.) 
Then we define a mapping $\mathbf 1\otimes R_{i,n}:B_0(\R_+^{n+1})\to B_0(\R_+^{i+1})$ by 
$$ \big(\mathbf 1\otimes R_{i,n} f^{(n+1)}\big)(t_1,\dots,t_{n+1})=\big(R_{i,n}
f^{(n+1)}(t_1,\cdot)\big)(t_2,\dots,t_{n+1}).$$
As easily seen, the mapping $\mathbf 1\otimes R_{i,n}$ is continuous and linear. 

We  define mappings
$D_{n-1,n}: B_0(\R_+^n)\to B_0(\R_+^{n-1})$, $D_{n-2,n}: B_0(\R_+^n)\to B_0(\R_+^{n-2})$, and $\mathcal I_{n-2,n}: B_0(\R_+^n)\to B_0(\R_+^{n-2})$ by 
\begin{align*}
\big(D_{n-1,n}f^{(n)}\big(t_1,\dots,t_{n-1}):&=f^{(n)}(t_1,t_1,t_2,\dots,t_{n-1}),\\
\big(D_{n-2,n}f^{(n)}\big(t_1,\dots,t_{n-2}):&=f^{(n)}(t_1,t_1,t_1,t_2,\dots,t_{n-2}),\\
\big(\mathcal I_{n-2,n}f^{(n)}\big)(t_1,\dots,t_{n-2}):&=\int_{t_1}^\infty
f^{(n)}(u,u,t_1,t_2,\dots,t_{n-2})\,du.  
\end{align*}
(For $n=2$, $\mathcal I_{0,2}f^{(2)}:=\int_{\R_+}f^{(2)}(u,u)\,du$.)

Let $\lambda\in\R$ and $\eta\ge0$ be fixed.
For $n\in\mathbb N$ and $i\in\mathbb N_0$, $i\le n$, we define continuous linear operators $R_{i,n}:B_0(\R_+^n)\to B_0(\R_+^i)$ by the following recursion formulas:
\begin{align}
&R_{1,1}=\mathbf 1,\quad R_{0,1}=\mathbf 0,\notag\\
&R_{2,2}=\mathbf 1,\quad R_{1,2}=-\lambda D_{1,2},\quad R_{0,2}=-\mathcal I_{0,2},\notag\\
&R_{i,n}=\mathbf 1\otimes R_{i-1,n-1}-\lambda R_{i,n-1}D_{n-1,n}-R_{i,n-2}\mathcal I_{n-2,n}-\eta R_{i,n-2}D_{n-2,n},\quad n\ge2.\label{tye6e}
\end{align}
(In the above formula, we assume that $R_{i,n}=0$ if $i>n$.) Note hat $R_{n,n}=\mathbf 1$ for all $n\in\mathbb N$. 

For each $f^{(n)}\in B_0(\R_+^n)$, we define
\begin{equation}\label{fdyi}
\la R^{(n)}(\omega),f^{(n)}\ra:=\sum_{i=0}^n\la\omega^{\otimes i},R_{i,n}f^{(n)}\ra.\end{equation}

\begin{lemma}\label{xe7868or} Assume that the condition of Lemma \ref{xtse546}  is satisfied. Then, for each $f^{(n)}\in B_0(\R_+^n)$,
\begin{equation}\label{uyro7}
\la R^{(n)}(\omega),f^{(n)}\ra=\la P^{(n)}(\omega),f^{(n)}\ra,\end{equation}
the equality in $L^2(\tau)$.
\end{lemma}

\begin{remark}
Since $\la P^{(n)}(\omega),f^{(n)}\ra$ depends only on the restriction of the function $f^{(n)}$ to $T_n$, formula \eqref{uyro7} means that 
the element of $L^2(\tau)$ given by formula \eqref{fdyi} also depends only on the restriction of the function $f^{(n)}$ to $T_n$.
However, this statement is not true about each individual term of the sum on the right hand side of \eqref{fdyi}. Take, for example, the term corresponding to $i=n$, i.e., $\la\omega^{\otimes n},f^{(n)}\ra$.  As easily seen, for $n\ge4$, this monomial does depend on the values of the function $f^{(n)}$ outside $T_n$.
\end{remark}

\begin{remark}\label{ye7648o5680}
It follows from Lemma \ref{xe7868or} and formula \eqref{fdyi} that, for each $f^{(n)}\in B_0(\R_+^n)$
$$\la\omega^{\otimes n},f^{(n)}\ra=\la P^{(n)}(\omega),f^{(n)}\ra-\sum_{i=0}^{n-1}\la\omega^{\otimes i},R_{i,n}f^{(i)}\ra.$$
Assume that $f^{(n)}=0$ $m_n$-a.e.\ on $T_n$. Then $\la P^{(n)}(\omega),f^{(n)}\ra=0$, so that 
$$\la\omega^{\otimes n},f^{(n)}\ra=-\sum_{i=0}^{n-1}\la\omega^{\otimes i},R_{i,n}f^{(i)}\ra\in\mathbf P^{(n-1)}.$$
Note that the above monomial $\la\omega^{\otimes n},f^{(n)}\ra$ is not necessarily equal to 0 as an element of $L^2(\tau)$.
\end{remark}

\begin{proof}[Proof of Lemma \ref{xe7868or}]

As easily seen, it suffices to prove formula \eqref{uyro7} in the case where $f^{(n)}=h_1\otimes\dots\otimes h_n$ for some $h_1,\dots,h_n\in B_0(\R_+)$.
In this case, formula \eqref{uyro7} is obviously true for $n=1,2$. Furthermore, it follows from the definition of $\la R^{(n)}(\omega),f^{(n)}\ra$ that the following recursion relation holds, for $n\ge3$,
\begin{align}
&\la R^{(n)}(\omega),h_1\otimes\dotsm\otimes h_n\ra=\la\omega,h_1\ra \la P^{(n-1)}(\omega),h_2\otimes\dotsm\otimes h_n\ra\notag\\
&\quad-\lambda\la R^{(n-1)}(\omega),(h_1h_2)\otimes h_3\otimes\dotsm\otimes h_n\ra-
\la R^{(n-2)}(\omega), \mathcal I_{3,1}(h_1,h_2,h_3)\otimes h_4\otimes \dotsm\otimes h_n\ra\notag\\
&\quad -\eta\la R^{(n-2)}(\omega),(h_1h_2h_3)\otimes h_4\otimes  \dotsm\otimes h_n\ra,\label{f5p90}
\end{align}
where the mapping $\mathcal I:B_0(\R_+)^3\to B_0(\R_+)$ is given by formula
\eqref{su5wsu5}. From here and Lemma~\ref{xtse546}  the statement follows by induction on $n$. 
\end{proof}

\begin{lemma}\label{tr78759p}
Assume that the condition of Lemma \ref{xtse546}  is satisfied. Then $\mathbf{CP}=\mathbf{OCP}$, i.e., the corresponding orthogonal polynomials belong to the Meixner class.
\end{lemma}

\begin{proof}
For each $n\in\mathbb N$, we consider the topology on $C_0(\R_+^n)$ that is induced by the topology on $B_0(\R_+)$. Thus, a sequence $(h_n)_{n=1}^\infty$  converges to a function $h$ in $C_0(\R_+)$ if all functions $h_n$ vanish outside a compact set in $\R_+$ and $\sup_{t\in\R_+}|h_n(t)-h(t)|\to0$ as $n\to\infty$.
For $n=0$, we will also set $C_0(\R_+^n):=\R$.

Let $n\in\mathbb N$ and $i=0,1,\dots,n$. 
If $R_{i,n}:C_0(\R_+^n)\to C_0(\R_+^i)$ is a continuous linear operator, then so is the mapping $\mathbf 1\otimes R_{i,n}:C_0(\R_+^{n+1})\to C_0(\R_+^{i+1})$.
Hence, we easily conclude that the (restrictions of the) operators $R_{i,n}$ defined by formula \eqref{tye6e} are continuous linear operators acting from 
$C_0(\R_+^n)$ to $C_0(\R_+^i)$, respectively. Therefore, by \eqref{fdyi}, for each $f^{(n)}\in C_0(\R_+^n)$, $\la R^{(n)}(\omega),f^{(n)}\ra\in\mathbf{CP}$. Now, Lemma~\ref{xe7868or} implies that $\la P^{(n)}(\omega),f^{(n)}\ra\in\mathbf{CP}$. 

Let $g^{(n)}\in C_0(T_n)$. Choose any $f^{(n)}\in C_0(\R_+^n)$ such that the restriction of $f^{(n)}$ to $T_n$ is equal to $g^{(n)}$. Then, by the proved above
$$\la P^{(n)}(\omega),g^{(n)}\ra=\la P^{(n)}(\omega),f^{(n)}\ra\in\mathbf{CP},$$
hence $\mathbf{OCP}\subset\mathbf{CP}$.

Let us now prove the inverse inclusion. 
Since $R_{n,n}=\mathbf 1$ for each $n$, it easily follows from \eqref{fdyi} and
Lemma~\ref{xe7868or} by induction on $n$ that, for each $f^{(n)}\in C_0(\R_+^n)$,
$$ \la \omega^{\otimes n},f^{(n)}\ra=\sum_{i=0}^n\la P^{(i)}(\omega),L_{i,n}f^{(n)}\ra,$$
where $L_{i,n}:C_0(\R_+^n)\to C_0(\R_+^i)$ are continuous linear operators. 
Let $g^{(i)}\in C_0(T_i)$ denote the restriction of the  function $L_{i,n}f^{(n)}$ to $T_i$. Then
$$ \la \omega^{\otimes n},f^{(n)}\ra=\sum_{i=0}^n\la P^{(i)}(\omega),g^{(i)}\ra.$$
Hence, $\mathbf{CP}\subset\mathbf{OCP}$.
\end{proof}

\begin{proof}[Proof of Theorem \ref{dre6i4e}]
By Lemmas \ref{xtse546} and \ref{tr78759p}, it remains to prove that, if 
$\mathbf{CP}=\mathbf{OCP}$, then the condition of Lemma \ref{xtse546} is satisfied. 

So we assume $\mathbf{CP}=\mathbf{OCP}$. Let $f^{(n)}\in C_0(T_n)$. Then
$\la P^{(n)}(\omega),f^{(n)}\ra\in\mathbf{CP}$. Since $\mathbf{CP}$ is an algebra under multiplication of two continuous polynomials, we conclude that, for each $h\in C_0(\R_+)$, $\la\omega,h\ra \la P^{(n)}(\omega),f^{(n)}\ra\in\mathbf{OCP}$. Hence, by Theorem~\ref{yd7i649}, there exist continuous functions $g^{(n)}\in C_0(T_n)$ and $g^{(n-1)}\in C_0(T_{n-1})$ such that
\begin{align}
&B^-(h)f^{(n)}=g^{(n-1)}\ \text{$m_{n-1}$-a.e.}\label{yufr767o}\\
&B^0(h)f^{(n)}=g^{(n)}\ \text{$m_n$-a.e.}\label{yrde6u4e7i}
\end{align}
If $a_k=0$ for all $k\in\mathbb N$, we set $\eta=0$ and $\lambda=b_0$, and the condition of Lemma \ref{xtse546} is satisfied. So, we only have to consider the case where $a_1>0$. Set $\eta=a_1$. Using the construction of the measure $m_n$, the definition of the operator $B^-(h)$ and formula \eqref{yufr767o}, we get by induction on $k$ that $a_k=\eta$ for all  $k\in\mathbb N$. But this also implies that $c_k>0$ for all $k\in\mathbb N$.
Now set $\lambda=b_0$. Using the definition of $B^0(h)$ and \eqref{yrde6u4e7i}, we deduce from \eqref{yrde6u4e7i} by induction on $k$ that $b_k=\lambda$ for all $k\in\mathbb N_0$.  
\end{proof}

\begin{proof}[Proof of Corollary \ref{t896p}] Immediate from Lemma \ref{xe7868or} and formula \eqref{f5p90}.
\end{proof}

\begin{center}
{\bf Acknowledgements}\end{center}
E.L. acknowledges the financial support of 
the SFB 701 ``Spectral
structures and topological methods in mathematics'', Bielefeld University.
E.L. is grateful to the Mathematical Institute of Wroclaw University for their hospitality and financial support during E.L.'s stay at the Institute. The authors would like to thank Marek Bo\.zejko and Janusz Wysocza\'nski for numerous useful discussions.

\end{document}